\documentclass[11pt]{article}
\usepackage[margin=1in]{geometry} 
\usepackage{amssymb,amsfonts,amsmath,bbm,mathrsfs,stmaryrd}
\usepackage{xcolor}
\usepackage{url}

\usepackage[colorlinks,
             linkcolor=black!75!red,
             citecolor=blue,
             pdftitle={Relative Fourier transforms and expectations on coideal subalgebras},
             pdfauthor={Alexandru Chirvasitu},
             pdfproducer={pdfLaTeX},
             pdfpagemode=None,
             bookmarksopen=true
             bookmarksnumbered=true]{hyperref}

\usepackage{tikz}
\usetikzlibrary{arrows,decorations.pathreplacing,decorations.markings,shapes.geometric,through,fit,shapes.symbols,positioning}

\usepackage[amsmath,thmmarks,hyperref]{ntheorem}
\usepackage{cleveref}

\crefname{section}{Section}{Sections}
\crefformat{section}{#2Section~#1#3} 
\Crefformat{section}{#2Section~#1#3} 

\crefname{subsection}{\S}{\S\S}
\crefformat{subsection}{#2\S#1#3} 
\Crefformat{subsection}{#2\S#1#3} 

\theoremstyle{plain}

\newtheorem{lemma}{Lemma}[section]
\newtheorem{proposition}[lemma]{Proposition}
\newtheorem{corollary}[lemma]{Corollary}
\newtheorem{theorem}[lemma]{Theorem}

\theoremstyle{nonumberplain}
\newtheorem{theoremN}{Theorem}

\theoremstyle{plain}
\theorembodyfont{\upshape}
\theoremsymbol{\ensuremath{\blacklozenge}}

\newtheorem{definition}[lemma]{Definition}

\newtheorem{remark}[lemma]{Remark}

\crefname{definition}{definition}{definitions}
\crefformat{definition}{#2definition~#1#3} 
\Crefformat{definition}{#2Definition~#1#3} 

\crefname{ex}{example}{examples}
\crefformat{example}{#2example~#1#3} 
\Crefformat{example}{#2Example~#1#3} 

\crefname{remark}{remark}{remarks}
\crefformat{remark}{#2remark~#1#3} 
\Crefformat{remark}{#2Remark~#1#3} 

\crefname{convention}{convention}{conventions}
\crefformat{convention}{#2convention~#1#3} 
\Crefformat{convention}{#2Convention~#1#3}

\crefname{lemma}{lemma}{lemmas}
\crefformat{lemma}{#2lemma~#1#3} 
\Crefformat{lemma}{#2Lemma~#1#3} 

\crefname{proposition}{proposition}{propositions}
\crefformat{proposition}{#2proposition~#1#3} 
\Crefformat{proposition}{#2Proposition~#1#3} 

\crefname{corollary}{corollary}{corollaries}
\crefformat{corollary}{#2corollary~#1#3} 
\Crefformat{corollary}{#2Corollary~#1#3} 

\crefname{theorem}{theorem}{theorems}
\crefformat{theorem}{#2theorem~#1#3} 
\Crefformat{theorem}{#2Theorem~#1#3} 

\crefname{assumption}{assumption}{Assumptions}
\crefformat{assumption}{#2assumption~#1#3} 
\Crefformat{assumption}{#2Assumption~#1#3} 

\crefname{equation}{}{}
\crefformat{equation}{(#2#1#3)} 
\Crefformat{equation}{(#2#1#3)}

\theoremstyle{nonumberplain}
\theoremsymbol{\ensuremath{\blacksquare}}

\newtheorem{proof}{Proof}
\newtheorem{proof of main}{Proof of \Cref{th.main}}
\newtheorem{proof of Hall}{Proof of \Cref{th.Hall}}
\newtheorem{proof of sbgp}{Proof of \Cref{th.sbgp}}

\newtheorem{proof of res1}{Proof of \Cref{th.res1}}
\newtheorem{proof of res2}{Proof of \Cref{th.res2}}
\newtheorem{proof of key}{Proof of \Cref{pr.key}}

\newtheorem{proof of le.aux}{Proof of \Cref{le.aux}}
\newtheorem{proof of s2}{Proof of \Cref{th.s2}}

\newcommand\bC{{\mathbb C}}

\newcommand\bR{{\mathbb R}}

\newcommand\cC{{\mathcal C}}
\newcommand\cD{{\mathcal D}}
\newcommand\cF{{\mathcal F}}

\newcommand\cM{{\mathcal M}}
\newcommand\cO{{\mathcal O}}
\newcommand\cV{{\mathcal V}}

\DeclareMathOperator{\id}{id}
\DeclareMathOperator{\End}{\mathrm{End}}

\newcommand{\define}[1]{{\em #1}}

\newcommand{\cat}[1]{\textsc{#1}}

\newcommand{\qedhere}{\mbox{}\hfill\ensuremath{\blacksquare}}

\newcommand{\ol}[1]{\overline{#1}}

\renewcommand{\square}{\mathrel{\Box}}

\title{Relative Fourier transforms and expectations on coideal subalgebras}
\author{Alexandru Chirvasitu\footnote{SUNY at Buffalo, \url{achirvas@buffalo.edu}}}

\begin{document}

\date{}

\maketitle

\begin{abstract}
For an algebraic compact quantum group $H$ we establish a bijection between the set of right coideal $*$-subalgebras $A\to H$ and that of left module quotient $*$-coalgebras $H\to C$. It turns out that the inclusion $A\to H$ always splits as a map of right $A$-modules and right $H$-comodules, and the resulting expectation $E:H\to A$ is positive (and lifts to a positive map on the full $C^*$ completion on $H$) if and only if $A$ is invariant under the squared antipode of $H$.  

The proof proceeds by Tannaka-reconstructing the coalgebra $C$ corresponding to $A\to H$ by means of a fiber functor from $H$-equivariant $A$-modules to Hilbert spaces, while the characterization of those $A\to H$ which admit positive expectations makes use of a Fourier transform turning elements of $H$ into functions on $C$.  
\end{abstract}


\vspace{.5cm}


\tableofcontents

\section*{Introduction}

The common setting for the results of the paper will be that of Hopf algebras $H$ regarded as non-commutative analogues of algebras of (particularly well-behaved) functions on compact or algebraic groups. Adopting this perspective consistently then dictates that a coideal subalgebra $\iota:A\to H$ (i.e. a subalgebra satisfying
\begin{equation*}
  \Delta(a)\in A\otimes H,\ \forall a\in A
\end{equation*}
for the comultiplication $\Delta:H\to H\otimes H$) then ought to be regarded as the algebra of functions on a ``quantum homogeneous space'' of the quantum group attached to $H$.

The above heuristic is one possible source of interest in the issue of whether or not $H$ is faithfully flat over $A$ (or dually, faithfully {\it co}flat over a quotient left module coalgebra). The problem has attracted a non-trivial amount of interest in the Hopf algebra literature; see e.g. \cite{Tak79,schn,mw,scha,scha1,skr,ff} and the numerous references therein.

For a quick illustration of how (co)flatness relates to homogeneous spaces classically (i.e. when $H$ is commutative) recall for instance the following remark made in passing in the introduction of \cite{Tak79}: if $N\subseteq G$ is a closed embedding of linear algebraic groups, then the quotient stack $G/N$ is representable by an affine scheme (the homogeneous space we have been referring to) if and only if the Hopf algebra $\cO(G)$ of regular functions on the scheme $G$ is faithfully coflat over its Hopf algebra quotient $\cO(G)\to \cO(N)$. The rough conclusion the reader should draw from this is that (co)flatness has something to do with homogeneous spaces being, vaguely speaking, well behaved.

Linear algebraic groups $G$ as above that in addition happen to be semisimple and defined over the complex numbers can be recast as analytic objects: $G$ has a maximal compact subgroup $K$; moreover, $\cO(G)$ admits a $*$-structure under which it is precisely the algebra of {\it representative functions} on $K$ (i.e. the algebra of matrix coefficients of finite-dimensional unitary representations $K\to U_n$) equipped with complex conjugation.

We work here with Hopf algebras amenable to the same type of dual treatment, being analogous both to algebras of regular functions on a linear algebraic group scheme and algebras of representative functions on a compact group. These are the {\it CQG algebras} of \cite{DijKoo94}, meant to recapture the theory of compact quantum groups initiated in \cite{Wor87} in a purely algebraic setting. We recall the details in \Cref{subse.prel_cqg}, observing for now that our Hopf algebras $H$ will be (at least) cosemisimple and equipped with a $*$-structure, just as $\cO(G)$ in the preceding paragraphs.

When working with coideal subalgebras $\iota:A\to H$ it is natural in this $*$-algebraic setting to restrict attention to those $A$ that inherit a $*$-structure from that of $H$. One of the main results of the paper then simply reads as follows (see \Cref{cor.cncl,cor.left-fl})

\begin{theoremN}
  A CQG algebra is left and right faithfully flat over any coideal $*$-subalgebra. 
\end{theoremN}

Furthermore, an inclusion $\iota$ of a right coideal $*$-subalgebra automatically splits as a right $A$-module and right $H$-comodule map via $E:H\to A$ (`splits' as in $E\circ\iota=\mathrm{id}_A$). The intuition here is that
\begin{equation*}
  E:\text{functions on the quantum group }G\to \text{ functions on its quantum homogeneous space }X
\end{equation*}
is ``integration along the fibers of $G\to X$'' (see \Cref{re.e} for a more precise take on this). As the word `integration' is meant to suggest, the topic will serve as a bridge between the algebraic (and category theoretic) discussion in \Cref{se.Tannaka,se.corr} and the operator algebraic framework that \Cref{se.exp} is naturally housed in.

In the compact quantum group  literature one typically completes $H$ with respect to a $C^*$ norm, and similarly for $A$. It is then of interest, in this analytic setup, whether or not $E:H\to A$ as above extends to an expectation in the $C^*$-algebraic sense (see e.g. \cite[Definition III.3.3]{tak1} or \cite[Definition IX.4.1]{tak2}):

\begin{itemize}
\item $E|A=\mathrm{id}_A$ (already the case);
\item $(h|A)\circ E = h$ (already the case);  
\item $\|E\|\le 1$. 
\end{itemize}

The latter norm condition depends on the $C^*$ norms under which $H$ is being completed. Here, we always complete $H$ with respect to its maximal $C^*$ norm and $A$ with respect to its subspace topology; we denote the completions by $H_u$ and $A_u$ respectively. It turns out that if $E$ does indeed lift to a $C^*$ expectation in the above sense then it is {\it completely positive}:

\begin{equation}\label{eq:pos}
\sum_{i,j}a_i^*E(x_i^*x_j)a_j\ge 0 
\end{equation}
in $A_u$ for every choice of finite tuples $(a_i)_{i=1}^n\subset A$ and $(x_i)_{i=1}^n\subset H_u$. 

Conversely, if our $E:H\to A$ is positive in the sense that \Cref{eq:pos} holds in $A_u$ then it extends to a $C^*$ expectation. This type of extensibility is of interest, for instance, in the literature on {\it idempotent states} on (locally) compact quantum groups \cite{fs1,fs2,fsr,ss1,ss2}. Idempotent states are, just as the name suggests, non-commutative analogues of idempotent measures on classical locally compact groups. For the latter, idempotent measures are Haar states on closed subgroups; in the quantum setting this is not quite the case, but the classification achieved in \cite[Theorem 1]{ss2} and earlier work cited therein puts idempotent states in bijection with coideal subalgebras $A$ admitting an expectation $E:H\to A$ in the sense we have been discussing.

The paper is organized as follows.

\Cref{se.prel} is devoted to recalling some of the conventions and auxiliary machinery used in the sequel, including compact quantum groups in their algebraic guise as covered in \cite{DijKoo94} and a brief recollection of Tannaka reconstruction (e.g. \cite{Sch92,Par}).

In \Cref{se.Tannaka} we consider right coideal subalgebras $\iota:A\to H$ and recover the resulting left module quotient coalgebra $C=H/HA^+$ via Tannaka reconstruction from a category of $H$-equivariant right $A$-modules (see \Cref{th.Tannaka_quot}). This will be useful later on, when we specialize to the case of CQG algebras and their right coideal $*$-subalgebras. Moreover, \Cref{cor.Tannaka_quot_expect} will later provide an expectation $E:H\to A$ splitting the inclusion $\iota$; this map will be the focus of the last section.  

\Cref{se.corr} is placed within this more restrictive context of CQG algebras. Its main result, \Cref{th.corr2}, argues that attaching $C=H/HA^+$ to $A$ as above implements a Galois correspondence between the lattice of right coideal $*$-subalgebras of $H$ and its left module quotient $*$-coalgebras. 

Finally, the main result (\Cref{th.s2}) of \Cref{se.exp} answers the question of when the expectation $E:H\to A$, whose existence is ensured by part (1) of \Cref{cor.cncl}, is positive in the sense of \Cref{eq:pos}. As explained above, this positivity is of interest in the study of idempotent states on compact quantum groups, and is the kind of behavior one requires of $E$ in order to transport the discussion the setting of operator algebras by taking $C^*$ (or von Neumann) completions. 

Part of the proof of \Cref{th.s2} relies on the Fourier transform introduced in \cite[Definition 3.10]{ff} and recalled in \Cref{def.four}. This is a map from $H$ to a space of functionals on $C$, acting as a non-commutative analogue of the Fourier transform on a locally compact abelian group: it turns elements of $H$, which are functions on our ``compact quantum group'' $G$, into functionals on $C$, i.e. intuitively functions on a quantum homogeneous space of the Pontryagin dual $\widehat{G}$.  

\subsection*{Acknowledgements}

The author was partially supported by NSF grants DMS-1565226 and DMS-1801011. 

I would also like to thank the anonymous referee for a very careful reading and invaluable suggestions for the improvement of the initial draft.

\section{Preliminaries}\label{se.prel}

All algebraic objects in this paper are vector spaces over the complex numbers and similarly, most categories we work with are $\bC$-linear. Algebras are always assumed to be unital (unless specified otherwise) and associative  and dually, coalgebras are counital and coassociative. For general material on coalgebras and Hopf algebras we refer to \cite{Swe69,Mon93,KliSch97,Rad12}. 

For coalgebras $C$ we use Sweedler notation (\cite[$\S$1.2.2]{KliSch97}) for comultiplication, as in $C\ni x\mapsto x_1\otimes x_2\in C\otimes C$. Similarly, we write $V\ni v\mapsto v_0\otimes v_1\in V\otimes C$ to denote a right $C$-comodule structure on $V$.

The letter $\cM$ with adorned corners denotes a category of (co)modules, perhaps with additional structure: Upper right corner means right comodule, lower left corner means left modules, etc. For instance, $\cM^C$ is the category of right comodules over the coalgebra $C$, and similarly, ${_A}\cM$ is the category of left $A$-modules.

Multiple decorations signify multiple structures that play well together, mostly in a way that lends itself to intuition. For instance, if $A$ is a right comodule-algebra over a Hopf algebra $H$, then $\cM_A^H$ is the category of \define{$H$-equivariant $A$-modules}: Its typical object $M$ is a right $A$-module and right $H$-comodule such that the module structure map $M\otimes A\to M$ is a morphism in the monoidal category $\cM^H$.  

We reserve an upper left $f$ to indicate finite-dimensional objects, as in ${^f}\cM^C$ (for finite-dimensional right $C$-comodules). We never use left comodules, so this will not clash with the other decorations.  

A $\bC$-linear category is \define{semisimple} if the endomorphism complex algebra of any object is semisimple and finite-dimensional. 

Although we work mostly algebraically, we will make use of operator-algebraic language. All we need is covered e.g. in \cite{tak1}. The terms `von Neumann algebra' and $W^*$-algebra are treated as synonymous in this paper. 

The symbol `$\otimes$' means different things depending on the nature of the tensorands. Mostly, it is simply the ordinary tensor product of vector spaces (over the complex numbers). When appearing between two $C^*$-algebras, it is the \define{minimal}, or \define{injective} tensor product \cite[IV.4.8]{tak1}. Similarly, when appearing between von Neumann algebras it denotes the \define{spatial} or $W^*$ tensor product \cite[IV.5.1]{tak1}. 

In all cases tensoring produces an object of the same nature as the tensorands: If $A$ and $B$ are $C^*$-algebras then so is $A\otimes B$, and similarly in the $W^*$ case.

\subsection{CQG algebras and quantum groups}\label{subse.prel_cqg}

For background on compact quantum groups we refer to \cite[Chapter 11]{KliSch97} or to the survey \cite{KusTus99}. 

Recall first that a \define{Hopf $*$-algebra} is a complex $*$-algebra $A$ as well as a Hopf algebra such that the comultiplication $\Delta:A\to A\otimes A$ and the counit $\varepsilon:A\to\bC$ are $*$-algebra maps; we reserve the symbols $\Delta$ and $\varepsilon$ for the comultiplications and counits respectively.

CQG algebras are Hopf $*$-algebras that behave like algebras of representative functions on a compact group (i.e. linear combinations of matrix entries coming from finite-dimensional representations of the group), minus the commutativity. One way to formalize this is as follows:

\begin{definition}\label{def.compat}
Let $H$ be a Hopf $*$-algebra and $\rho:M\to M\otimes H$ a finite-dimensional right comodule. Denote by $M^*$ the dual of $M$ in the monoidal category $\cM^H$, and by $\overline{M}$ the complex conjugate space equipped with the $H$-comodule structure 
\begin{equation}\label{eq:conj_comod}
\begin{tikzpicture}[auto,baseline=(current  bounding  box.center)]
  \node (1) at (0,0) {$ M$};
  \node (2) at (3,0) {$ M\otimes H$};
  \node (3) at (6,0) {$ M\otimes H$.};
  \draw[->] (1) to node {$ \rho$} (2);
  \draw[->] (2) to node {$ \id\otimes *$} (3);
\end{tikzpicture}
\end{equation} 

An inner product $\langle -| -\rangle$ on $M$ can be thought of as a map from the complex conjugate space $\overline{M}$ to the dual $M^*$. We say that $\langle -| -\rangle$ is \define{compatible} with the comodule structure if this map is one of comodules. 
\end{definition}

\begin{definition}\label{def.cqg}
A \define{CQG algebra} is a Hopf $*$-algebra for which every finite-dimensional right comodule admits a compatible inner product.  
\end{definition}

Comodules admitting a compatible inner product are sometimes called \define{unitarizable}, while those equipped with a specific such inner product are \define{unitary}. We denote the category of unitary comodules by ${_u}\cM^H$, with an extra `f' in the upper left corner to indicate the finite-dimensional ones. 

Commutative CQG algebras are precisely algebras of representative functions on compact groups. Just as in this classical case, a CQG algebra has a distinguished functional, the \define{Haar state} \cite[$\S$11.3.2]{KliSch97}, behaving as a bi-invariant probability measure on the corresponding ``quantum group'': If $h:H\to\bC$ is the Haar state of the CQG algebra and $H\ni x\mapsto x_1\otimes x_2\in H\otimes H$ is the comultiplication of $H$, then
\begin{equation}\label{eq:biinv}
 h(x_1)x_2 = h(x)1 = h(x_2)x_1,\ \forall x\in H. 
\end{equation}
This is equivalent to $\varphi h=h\varphi = \varphi(1)h$ for all functionals $\phi: H\to \bC$, where the space of functionals is made into an associative algebra as usual, using the comultiplication:  
\begin{center}
\begin{tikzpicture}[auto]
  \node (1) at (0,0) {$ H$};
  \node (2) at (3,0) {$ H\otimes H$};
  \node (3) at (6,0) {$ \bC$};
  \draw[->] (1) to node {$\scriptstyle \Delta$} (2);
  \draw[->] (2) to node {$\scriptstyle \varphi\otimes\psi$} (3);
  \draw[->,bend right=30] (1) to node [swap] {$\scriptstyle \varphi\psi$} (3);
\end{tikzpicture}
\end{center}  
for $\varphi,\psi$ in the dual $H^*$. 

The existence of a unital functional $h$ biinvariant in the sense of \Cref{eq:biinv} characterizes \define{cosemisimple} Hopf algebras, i.e. those whose categories of finite-dimensional comodules are semisimple. CQG algebras satisfy an even stronger condition: The Haar state $h$ is \define{positive}, in the sense that $h(x^*x)\ge 0$ for all $x\in H$ and {\it faithful} in that the preceding inequality is an equality only for $x=0$. 

The positivity of the Haar functional allows us to pass from the purely algebraic setup above to the analytic one: 

Completing $H$ in the GNS representation for $h$ \cite[I.9.15]{tak1} with respect to either the norm or the weak$^*$ topology produces a $C^*$ and respectively a von Neumann algebra. We denote them by $C^*_{r}(H)$ and $W^*_{r}(H)$ respectively (`$r$' for `reduced'). 

The comultiplication $\Delta_H:H\to H\otimes H$ lifts to a morphism $\Delta_{\overline{H}}:\overline{H}\to\overline{H}\otimes\overline{H}$ in the appropriate category ($C^*$ or $W^*$-algebras) when $\overline{H}$ is either $C^*_{r}(H)$ or $W^*_{r}(H)$ (recall that the tensor product also changes meaning between the two cases).

$C^*_{r}(H)$ is the smallest $C^*$ completion of $H$ that will admit a continuous extension of $h$ to a state (in the usual sense for $C^*$-algebras: continuous positive, unital functional). There is also a \define{largest} $C^*$-completion for $H$, obtained by assigning to each element $x\in H$ the largest possible norm of $x$ in a $*$-representation of $H$ on a Hilbert space \cite[$\S$4]{DijKoo94}. We denote this latter completion by $C^*_f(H)$ (for `full'). As before, the comultiplication of $H$ lifts to a $C^*$-algebra map $C^*_f(H)\to C^*_f(H)\otimes C^*_f(H)$. 

$C^*_r(H)$ and $C^*_f(H)$ equipped with their respective comultiplications are both $C^*$-algebraic compact quantum groups, as first introduced in \cite{Wor87}. We do not recall the definition of the abstract notion here, since much of the analysis in this paper is carried out in the purely algebraic framework. We remind the reader only that every $C^*$-algebraic compact quantum group is a unital $C^*$-algebra equipped with a comultiplication (satisfying some additional properties), and that every such algebra has a unique dense CQG subalgebra; see e.g. \cite[Theorem 3.1.7]{KusTus99}. This CQG subalgebra is to be thought of as the algebra of representative functions sitting inside the algebra of all continuous functions on the quantum group; for this reason, the CQG algebra contains all of the relevant representation-theoretic information.

\subsection{Tannaka reconstruction}\label{subse.prel_Tannaka}

The main references for this subsection are \cite{Sch92} an Chapter 3 of \cite{Par}. As there is a wealth of material on the topic in the literature, we give only a brief and somewhat informal account, covering only what is needed below. 

The input is a functor $\omega:\cC\to \cV$ which one wants to think of as the forgetful functor from some category of (co)modules to that of vector spaces.  More precisely $\cV$ is always a (typically symmetric) monoidal category, and one looks for a canonical construction of (say) an algebra $A$ internal to $\cV$ so that $\omega$ factors up to natural isomorphism as 
\begin{center}
\begin{tikzpicture}[auto,baseline=(current  bounding  box.center)]
  \node[] (1) at (0,0) {$ \cC$};
  \node[] (2) at (4,0) {$ \cM_A$};
  \node[] (3) at (2,-1) {$ \cV$};
  \draw[->,bend left=10] (1) to node {} (2);
  \draw[->,bend right=10] (1) to node[below] {$\scriptstyle \omega$} (3); 
  \draw[->,bend left=10] (2) to node {$\scriptstyle \text{forget}$} (3);
\end{tikzpicture}
\end{center}
One way to get a candidate algebra $A$ is by demanding that it act on $\omega$ universally (on the right, say). What this means is one considers the functor $\cat{act}:\cV^{\mathrm{op}}\to\cat{Set}$ defined by
\[
 \cV\ni V\mapsto \text{natural transformations }\omega\otimes V\to \omega, 
\] 
and looks for an object $A\in\cV$ that represents it, in the sense that $\cat{act}\cong \hom_\cV(\bullet,A)$. 

It turns out that for various technical reasons, this works better dually: 

Suppose $\cV$ is cocomplete, the tensor products play well with the colimits, and $\omega$ lands inside the full subcategory of rigid objects (i.e. those which admit a left dual). Then the (this time covariant) functor  
\[
 \cat{coact}:\cV\to\cat{Set},\quad V\mapsto \text{natural transformations } \omega\to \omega\otimes V
\]
is representable as $\hom_{\cV}(C,\bullet)$ for some object $C\in\cV$, and $C$ is naturally a coalgebra which factors $\omega$ as 
\begin{center}
\begin{tikzpicture}[auto,baseline=(current  bounding  box.center)]
  \node[] (1) at (0,0) {$ \cC$};
  \node[] (2) at (4,0) {$ \cM^C$};
  \node[] (3) at (2,-1) {$ \cV$};
  \draw[->,bend left=10] (1) to node {} (2);
  \draw[->,bend right=10] (1) to node[below] {$\scriptstyle \omega$} (3); 
  \draw[->,bend left=10] (2) to node {$\scriptstyle \text{forget}$} (3);
\end{tikzpicture}
\end{center}
The sloppiness in the above statement will not cause any difficulties below. For us, $\cV$ will always be some well-behaved category for which everything works out: complex vector spaces or perhaps complex $*$-vector spaces (i.e. complex vector spaces equipped with a conjugate-linear involutive automorphism). 

The coalgebra $C$ associated to $\omega$ in the manner described above is sometimes called the \define{coendomorphism coalgebra} of $\omega$, and is denoted by $\cat{coend}(\omega)$. 

The association $\omega\mapsto \cat{coend}(\omega)$ is functorial in the sense that if $\cC\to \cD$ is a functor making 
\begin{center}
\begin{tikzpicture}[auto,baseline=(current  bounding  box.center)]
  \node[] (1) at (0,0) {$ \cC$};
  \node[] (2) at (2,0) {$ \cD$};
  \node[] (3) at (1,-1) {$ \cV$};
  \draw[->,bend left=10] (1) to node {} (2);
  \draw[->,bend right=10] (1) to node[below] {$\scriptstyle \omega$} (3); 
  \draw[->,bend left=10] (2) to node {$\scriptstyle \eta$} (3);
\end{tikzpicture}
\end{center}
commute up to isomorphism, then there is a morphism $C\to D$ between the coendomorphism coalgebras of $\omega$ and $\eta$ respectively whose associated scalar corestriction functor $\cM^C\to\cM^D$ makes the diagram
\begin{equation}\label{eq:3d_diagram}
  \begin{tikzpicture}[baseline=(current  bounding  box.center),anchor=base,cross line/.style={preaction={draw=white,-,line width=6pt}}]
    \path (0,0) node (1) {$\scriptstyle \cC$} +(3,0) node (2) {$\scriptstyle \cD$} +(2,-.5) node (3) {$\scriptstyle \cM^C$} +(5,-.5) node (4) {$\scriptstyle \cM^D$} +(3,-2) node (5) {$\scriptstyle \cV$}; 
    \draw[->] (1) -- (2);
    \draw[->] (1) -- (3);
    \draw[->] (2) -- (4);
    \draw[->] (1) to[bend right=30] node[pos=.5,auto,swap] {$\scriptstyle \omega$} (5);
    \draw[->] (3) to[bend right=20] (5);
    \draw[->] (4) to[bend left=30] (5);
    \draw[->] (2) -- (5) node[pos=.7,auto] {$\scriptstyle \eta$};
    \draw[->,cross line] (3) -- (4);
  \end{tikzpicture}
\end{equation}
commutative, or rather 2-commutative; there are natural isomorphisms filling in the triangles compatibly, which we have suppressed.

Additional structure on $\cC$ and structure-preserving conditions on $\omega$ translate to additional structure ``of the same kind'' on the coendomorphism coalgebra $C$. If $\cC$ and $\omega$ are monoidal, then $\cC$ is a bialgebra. If in addition $\cC$ is rigid (in which case $\omega$ will automatically preserve duals), then $C$ is a Hopf algebra. Moreover, the functoriality \Cref{eq:3d_diagram} holds with appropriate modifications (monoidal functors $\cC\to \cD$ induce morphisms of bialgebras, etc.).  

In the world of compact quantum groups one constantly deals with $*$-structures. To see what these correspond to from   Tannakian perspective, recall first the following notion.

\begin{definition}\label{def.*struc}
A \define{$*$-structure} on a $\bC$-linear category $\cC$ is a conjugate-linear, involutive contravariant functor $*:\cC\to \cC$ which is (naturally isomorphic to) the identity at the level of objects.  


A $\bC$-linear category equipped with a $*$-structure is a \define{$*$-category}. 
\end{definition}

Now suppose $\cC$ is a $*$-category and $\cV$ is the cocompletion of the category ${^f}\cV$ of finite-dimensional Hilbert spaces. ${^f}\cV$ is a $*$-category in an obvious way, with the $*$-functor taking adjoints of operators. Since we are assuming that $\omega$ takes values in ${^f}\cV$, it makes sense to require that it intertwine the $*$-structures (i.e. that it be a $*$-functor). If it does, then in the above setup one can show that $C=\cat{coend}(\omega)$ is naturally equipped with a conjugate-linear, involutive, comultiplication-reversing operation $*$. This makes the following definition relevant

\begin{definition}\label{def.*coalg}
A \define{$*$-coalgebra} is a complex coalgebra equipped with an involutive, conjugate-linear, comultiplication-reversing map. 
\end{definition}

\begin{remark}
  $*$-coalgebras as defined here are the $\circ$-coalgebras of \cite{afh}. 
\end{remark}

The typical example of a functor $\omega$ meeting all the requirements from the discussion preceding \Cref{def.*coalg} is the forgetful functor from the category ${_u}{^f}\cM^H$ of finite-dimensional unitary $H$-comodules for a CQG algebra $H$ to finite-dimensional Hilbert spaces. The operation making $H$ into a $*$-coalgebra is \define{not} the usual $*$ from the definition of a CQG algebra; that map reverses multiplication and preserves comultiplication. Rather, it is $H\ni x\mapsto (Sx)^*$ (this is automatically an involution for any Hopf $*$-algebra).

\section{Coideal subalgebras and equivariant modules}\label{se.Tannaka}

The general setup for this section is as follows: $H$ will be a Hopf algebra, $\iota:A\to H$ a right coideal subalgebra, and $\pi:H\to C$ a left module quotient coalgebra. 

There is a duality construction that gives rise to an inclusion $\iota$ given $\pi$ and vice versa: For fixed $\iota:A\to H$, define $H_A=H_\iota=H/HA^+$, where $A^+=\ker\varepsilon|_A$. Dually, given $\pi:H\to C$, set 
\[
 {^C}H={^\pi}H=\{ h\in H\ |\ \pi(h_1)\otimes h_2 = \pi(1)\otimes h\}.
\]
These constructions implement order-reversing maps
\begin{equation}\label{eq:Galois}
\begin{tikzpicture}[auto,baseline=(current  bounding  box.center)]
  \node[text width=4cm] (1) at (0,0) {set of right coideal subalgebras of $H$};
  \node[text width=4.5cm] (2) at (8,0) {set of quotient left module coalgebras of $H$};
  \draw[->] (1) .. controls (2,1) and (6,1) .. node{$A\mapsto H_A$} (2);
  \draw[<-] (1) .. controls (2,-1) and (6,-1) .. node[below] {${^C}H\mapsfrom H$} (2);
\end{tikzpicture}
\end{equation}
with respect to the standard orderings on the sets of subspaces and quotients of $H$ respectively.

Fix $\iota:A\to H$, and consider the corresponding quotient $\pi:H\to C=H_\iota$. The functors 
\begin{equation}\label{eq:adj}
\begin{tikzpicture}[auto,baseline=(current  bounding  box.center)]
  \node[] (1) at (0,0) {$\cM^H_A$};
  \node[] (2) at (6,0) {$\cM^C$};
  \draw[->] (1) .. controls (2,.5) and (4,.5) .. node{$M\mapsto M/MA^+$} (2);
  \draw[<-] (1) .. controls (2,-.5) and (4,-.5) .. node[below] {$N\square_CH \mapsfrom N$} (2);
\end{tikzpicture}
\end{equation}
constitute an adjunction, where the lower arrow indicates the \define{cotensor product} between the right $C$-comodule $N$ and the left $C$-comodule $H$. If $\rho: N\to N\otimes C$ is the comodule structure map, then $N\square_CH$ is defined as the equalizer
\begin{equation}\label{eq:cotensor}
\begin{tikzpicture}[auto,baseline=(current  bounding  box.center)]
  \node[] (1) at (0,0) {$N\square_CH$};
  \node[] (2) at (2,0) {$N\otimes H$};
  \node[] (3) at (6,0) {$N\otimes C\otimes H$.};
  \draw[->] (1) to (2);
  \draw[->] (2) .. controls (3,.5) and (5,.5) .. node{$\scriptstyle \rho\otimes\id_H$} (3);
  \draw[->] (2) .. controls (3,-.5) and (5,-.5) .. node[below] {$\scriptstyle \id_N\otimes (\pi\otimes\id_H)\circ\Delta$} (3);
\end{tikzpicture}
\end{equation}

If $H$ is, say, the algebra of regular functions on a linear algebraic group $G_H$ and $C$ is functions on a closed subgroup $G_C\le G_H$, then the lower arrow in \Cref{eq:adj} is the induction functor from $G_H$-representations to $G_C$ representations. In good situations, the adjunction \Cref{eq:adj} is an equivalence. In the setting of algebraic groups, for instance, this is the case whenever $G_H$ and $G_C$ are reductive. 

In general, technical conditions such as (co)flatness will ensure that the general setup has pleasant properties. Recall that $H$ is said to be \define{coflat} as a $C$-comodule if the endofunctor $\bullet\square_CH$ of $\cM^C$ defined by \Cref{eq:cotensor} is exact; it is \define{faithfully} coflat if the functor is exact and faithful. There is an obvious notion of coflatness for right comodules as well, and if $C$ is cosemisimple (i.e. its category of finite-dimensional comodules is semisimple) then all comodules, right or left, are coflat.

With this in place, we have for instance the following paraphrase of \cite[Theorem 1]{Tak79}:

\begin{theorem}\label{th.tak1}
Let $H$ be a Hopf algebra, $\iota:A\to H$ be a right coideal subalgebra, and $C=H_\iota$. 

The adjunction \Cref{eq:adj} is an equivalence if and only if $H$ is a faithfully coflat left $C$-comodule. This is the case whenever $H$ is a faithfully flat left $A$-module. \qedhere
\end{theorem}

We will also need a dual analogue of \Cref{th.tak1}, in order to address the problem of whether or not $H$ is faithfully flat over $A$. To this end, consider first the adjunction
\begin{equation}\label{eq:adj-bis}
\begin{tikzpicture}[auto,baseline=(current  bounding  box.center)]
  \node[] (1) at (0,0) {${}_A\cM$};
  \node[] (2) at (6,0) {${}^C_H\cM$};
  \draw[->] (1) .. controls (2,.5) and (4,.5) .. node{$M\mapsto H\otimes_AM$} (2);
  \draw[<-] (1) .. controls (2,-.5) and (4,-.5) .. node[below] {${}^{C}N \mapsfrom N$} (2);
\end{tikzpicture}
\end{equation}
obtained as the dual companion to \Cref{eq:adj}: replace module structures with comodule structures, etc. ${}^CN$ stands for
\begin{equation*}
  \{n\in N\ |\ n_{-1}\otimes n_0=\pi(1)\otimes n\},
\end{equation*}
where $n\mapsto n_{-1}\otimes n_0$ is the left $C$-comodule structure on $N$. We can then rephrase \cite[Theorem 2]{Tak79} slightly as follows.

\begin{theorem}\label{th.tak2}
Let $H$ be a Hopf algebra, $\pi:H\to C$ a left $H$-module quotient coalgebra, and $A={}^{\pi}H$.

The adjunction \Cref{eq:adj-bis} is an equivalence if and only if $H$ is a faithfully flat right $A$-module. This is the case whenever $H$ is a faithfully coflat right $C$-comodule. \qedhere  
\end{theorem}

In the next section we will show that an inclusion $\iota$ of a right coideal $*$-subalgebra in a CQG algebra gives rise to the following ideal situation: 

The coalgebra $C=H_\iota$ is cosemisimple, the coideal subalgebra we started out with is exactly ${^C}H$, and \Cref{eq:adj} is an equivalence. In order to prove this, we need a Tannakian characterization of the quotient coalgebra $H_\iota$.

Fixing a right coideal subalgebra $\iota:A\to H$, let ${^{ff}}\cM_A^H$ (for `finite free') be the full subcategory of $\cM_A^H$ consisting of objects of the form $V\otimes A$ for some $V\in{^f}\cM^H$. The main result of this section is

\begin{theorem}\label{th.Tannaka_quot}
Let $\cC={^{ff}}\cM_A^H$, and $\omega:\cC\to\cat{Vec}$ the restriction of the functor $\cM_A^H\to\cat{Vec}$ defined by $M\mapsto M/MA^+$. 

Then, the coalgebra $C=H_A$ can be identified with $\cat{coend}(\omega)$. 
\end{theorem}

\begin{remark}
Recall from \Cref{subse.prel_Tannaka} that in order to make rigorous sense of the Tannaka machinery we want $\omega$ to take values in \define{finite} vector spaces. This is clearly the case for $M\mapsto M/MA^+$ restricted to ${^{ff}}\cM_A^H$. 
\end{remark}

\begin{proof}
The morphism $H\to C$ will be reconstructed from the commutative triangle 
\begin{equation}\label{eq:Tannaka_quot_1}
  \begin{tikzpicture}[baseline=(current  bounding  box.center),anchor=base,cross line/.style={preaction={draw=white,-,line width=6pt}}]
    \path (0,0) node (1) {$ {^f}\cM^H$} +(4,0) node (2) {$ \cC$} +(2,-1) node (3) {$ \cat{Vec}$}; 
    \draw[->] (1) -- (2) node[pos=.5,auto]{$\scriptstyle V\mapsto V\otimes A$};
    \draw[->] (1) to[bend right=10] node[pos=.5,auto,swap] {$\scriptstyle \text{forget}$} (3);
    \draw[->] (2) to[bend left=10] node[pos=.5,auto] {$\scriptstyle \omega$} (3);   
  \end{tikzpicture}
\end{equation}
using the functoriality of $\omega\mapsto\cat{coend}(\omega)$ (in the form \Cref{eq:3d_diagram}) and the fact that the coendomorphism coalgebra of $\text{forget}$ in \Cref{eq:Tannaka_quot_1} is nothing but $H$ itself (see e.g. \cite[Theorem 3.8.4]{Par}). 

Denote $C_\omega=\cat{coend}(\omega)$. 

{\bf (a): Constructing a surjective $H$-module map $H\to C_\omega$.} Recall e.g. from \cite[Chapter 3, $\S$4]{Par} that $C_\omega$ is constructed as the coend $\int^\cC\omega^*\otimes \omega$. What this means, concretely, is that $C_\omega$ is put together by first taking the direct sum $\bigoplus_{x\in\cC}\omega(x)^*\otimes\omega(x)$ and then coequalizing diagrams of the form
\begin{equation}\label{eq:Tannaka_quot_2}
  \begin{tikzpicture}[baseline=(current  bounding  box.center),anchor=base,cross line/.style={preaction={draw=white,-,line width=6pt}}]
    \path (0,0) node (1) {$ \omega(y)^*\otimes\omega(x)$} +(4,1) node (2) {$ \omega(y)^*\otimes\omega(y)$} +(4,-1) node (3) {$ \omega(x)^*\otimes\omega(x)$}; 
    \draw[->] (1) to[bend left=10] node[pos=.7,auto] {$\scriptstyle \id\otimes\omega(f)$} (2);
    \draw[->] (1) to[bend right=10] node[pos=.7,auto,swap] {$\scriptstyle \omega(f)^*\otimes \id$} (3);   
  \end{tikzpicture}
\end{equation}
for morphisms $f:x\to y$ in $\cC$ (i.e. modding out by the relations that dictate that the images of all $v\in \omega(y)^*\otimes\omega(x)$ through the two maps agree). The same construction returns $H$ when applied to the left hand arrow in \Cref{eq:Tannaka_quot_1}. 

Since the horizontal arrow in \Cref{eq:Tannaka_quot_1} induces a bijection at the level of objects, the coend construction just described produces the same coproduct $\bigoplus \omega(x)^*\otimes\omega(x)$ when applied to the forgetful functor on ${^f}\cM^H$ instead of $\omega$. However, in this case there might be fewer relations of the form \Cref{eq:Tannaka_quot_2}; this means that the horizontal arrow in \Cref{eq:Tannaka_quot_1} implements a surjection $\pi_\omega: H\to C_\omega$. 

Moreover, note that $\cC$ is a \define{module category} over the monoidal category ${^f}\cM^H$ in the sense of e.g. \cite[Definition 2.6]{Ost03}: For every $V\in{^f}\cM^H$ and $M\in \cC$ we can construct the  object $V\otimes M$ with tensor product $H$-coaction and right hand $A$-action, $(V,M)\mapsto V\otimes M$ is compatible with the associativity constraint of ${^f}\cM^H$, etc. (see loc. cit. for details). 

The horizontal arrow in \Cref{eq:Tannaka_quot_1} is compatible with the ${^f}\cM^H$-module category structures on the domain and codomain, which upon applying the Tannakian formalism translates to $\pi_\omega:H\to C_\omega$ being a map of left $H$-modules.

{\bf (b): $\pi_\omega$ annihilates $HA^+$.} Since we already know $\pi_\omega$ is a map of left $H$-modules, it is enough to show that it annihilates $A^+$. 

Let $i:V\le A$ be a finite-dimensional sub-$H$-comodule. The composition
\begin{equation}\label{eq:Tannaka_quot_2.5}
 \begin{tikzpicture}[baseline=(current  bounding  box.center),anchor=base,cross line/.style={preaction={draw=white,-,line width=6pt}}]
    \path (0,0) node (1) {$ V\otimes A$} +(3,0) node (2) {$ A\otimes A$} +(6,0) node (3) {$ A$}; 
    \draw[->] (1) to node[pos=.5,auto] {$\scriptstyle i\otimes \id$} (2);
    \draw[->] (2) to node[pos=.5,auto] {$\scriptstyle \text{multiply}$} (3);
    \draw[->] (1) to[bend right=20] node[pos=.5,auto,swap] {$ f$} (3);
  \end{tikzpicture}
\end{equation}  
is then a morphism in $\cC$, and setting $x=V\otimes A$ and $y=A$ in \Cref{eq:Tannaka_quot_2} (with $f$ as defined in \Cref{eq:Tannaka_quot_2.5}) produces the diagram 
\begin{equation}\label{eq:Tannaka_quot_2.7}
  \begin{tikzpicture}[baseline=(current  bounding  box.center),anchor=base,cross line/.style={preaction={draw=white,-,line width=6pt}}]
    \path (0,0) node (1) {$ V$} +(3,.5) node (2) {$ \bC$} +(3,-.5) node (3) {$ V^*\otimes V$}; 
    \draw[->] (1) to[bend left=10] node[pos=.7,auto] {$\scriptstyle \id\otimes\omega(f)$} (2);
    \draw[->] (1) to[bend right=10] node[pos=.7,auto,swap] {$\scriptstyle \omega(f)^*\otimes \id$} (3);   
  \end{tikzpicture}
\end{equation} 
The upper arrow is just the restriction of $\varepsilon$ to $V$, while the lower arrow followed by the canonical map $V^*\to V\to H$ (obtained from $V\to V\otimes H$ by moving the right hand $V$ to the left) is nothing but the inclusion of $V$ in $H$. This means that $H\to C_\omega$ annihilates $\ker(\varepsilon|_V)$, and the conclusion follows from the fact that $A$ is the union of its finite-dimensional sub-comodules.

{\bf (c): Identification of $C$ and $C_\omega$.} By construction, the functor $M\mapsto M/MA^+$ on $\cM_A^H$ factors through $\cM^C$ by \Cref{eq:adj}, so we have a 2-commutative diagram
\begin{equation}\label{eq:Tannaka_quot_3}
  \begin{tikzpicture}[baseline=(current  bounding  box.center),anchor=base,cross line/.style={preaction={draw=white,-,line width=6pt}}]
    \path (0,0) node (1) {$ \cC$} +(3,0) node (2) {$ {^f}\cM^C$} +(1.5,-1) node (3) {$ \cat{Vec}$}; 
    \draw[->] (1) -- (2);
    \draw[->] (1) to[bend right=10] node[pos=.5,auto,swap] {$\scriptstyle \omega$} (3);
    \draw[->] (2) to[bend left=10] node[pos=.5,auto] {$\scriptstyle \text{forget}$} (3);   
  \end{tikzpicture}
\end{equation}
which by the universality of the $\cat{coend}$ construction induces a map $C_\omega\to C$. Precomposing with the surjection $\pi_\omega:H\to C_\omega$ from above produces exactly $\pi:H\to C$, meaning that $C_\omega\to C$ is surjective. However, we showed above that the kernel of $\pi:H\to C$ is already annihilated by $\pi_\omega$. It follows that $C_\omega\to C$ is an isomorphism, and $\pi_\omega$ can be identified with $\pi$. 
\end{proof}

\begin{remark}\label{re.alt_thm}
With only minor modifications to the proof we could have substituted for ${^{ff}}\cM_A^H$ either the category ${^{fp}}\cM_A^H$ (for `finitely presented') of $H$-equivariant $A$-modules which admit a presentation by objects in ${^{ff}}\cM_A^H$, or the category $^{fg}\cM_A^H$ of those that are finitely generated as $A$-modules. 
\end{remark}

Before stating the following result, recall from \Cref{se.prel} that for us a semisimple $\bC$-linear category is one in which the endomorphism algebra of every object is semisimple and finite-dimensional.

\begin{corollary}\label{cor.Tannaka_quot}
If $\cC={^{ff}}\cM_A^H$ is semisimple, then $C=H_A$ is cosemisimple and the adjunctions \Cref{eq:adj,eq:adj-bis} are equivalences.
\end{corollary}
\begin{proof}
{\bf (1): $C$ is cosemisimple.} Let $\ol{\cC}$ be the idempotent completion of $\cC$ in the abelian category $\cM_A^H$: It is the full subcategory of $\cM_A^H$ consisting of all objects which are retracts of some object in $\cC$. More abstractly, $\ol{\cC}$ can be identified with the \define{Karoubi envelope} of $\cC$, i.e. the category obtained by universally splitting idempotent endomorphisms in $\cC$; see \cite[$\S$6]{BarMor06} for a brief recollection. 

Since $\cC$ is semisimple, it is easy to see that $\ol{\cC}$ is semisimple and abelian. To elaborate, recall for instance from the discussion in \cite[$\S$2.1]{mug1} that one way to phrase this claim is as the conjunction of the following conditions:
\begin{itemize}
\item $\ol{\cC}$ admits direct sums;
\item idempotent morphisms split in $\ol(\cC)$ (this is called ``having subobjects'' in \cite[$\S1$]{mug1}); 
\item there are objects $X_i\in \ol{\cC}$ such that
\begin{equation*}
  \dim \mathrm{Hom}(X_i,X_j) = \delta_{ij}
\end{equation*}
and for every two objects $Y,Z\in \ol{\cC}$ the canonical map
\begin{equation*}
  \bigoplus_i \mathrm{Hom}(Y,X_i)\otimes \mathrm{Hom}(X_i,Z) \to \mathrm{Hom}(Y,Z)
\end{equation*}
is an isomorphism.
\end{itemize}
The first two conditions obviously hold. As for the third, the $X_i$ will be the summands of arbitrary objects $X\in \ol{\cC}$ splitting the minimal idempotents of the finite product $\mathrm{End}(X)$ of matrix algebras. 

The functor $\omega:\cC\to\cat{Vec}$ from \Cref{th.Tannaka_quot} extends to $\ol{\cC}$ by the universality of the Karoubi envelope, and the extension $\ol{\omega}$ is exact by semisimplicity because $M\mapsto M/MA^+$ is right exact on $\cM_A^H$. Finally, $\omega$ is faithful by \Cref{le.faith} below, and hence so is $\ol{\omega}$. 

In conclusion, we have an exact, faithful functor from the abelian $\bC$-linear category $\ol{\cC}$ to ${^f}\cat{Vec}$. This implies (\cite[Theorem 2.2.8]{Sch92}) that the canonical functor $\ol{\cC}\to{^f}\cM^{\cat{coend}(\ol{\omega})}$ is an equivalence, and hence $\cat{coend}(\ol{\omega})$ is cosemisimple. We leave it to the reader to verify that passing to the idempotent completion does not affect coendomorphism coalgebras, so that $\cat{coend}(\ol{\omega})\cong C$.   

{\bf (2): The adjunction \Cref{eq:adj} is an equivalence.} We already know from the proof of part (1) that $M\mapsto M/MA^+$ is an equivalence of $\ol{\cC}$ onto ${^f}\cM^C$. I first claim that $\ol{\cC}$ is exactly the full subcategory ${^{fg}}\cM_A^H$ of $\cM_A^H$ consisting of objects that are finitely generated as $A$-modules. 

{\bf (a): Proof of the claim.} We have to prove that an arbitrary object $M\in{^{fg}}\cM_A^H$ is in $\ol{\cC}$. By finite generation, there is an epimorphism $p:V\otimes A\to M$ in $\cM_A^H$ for some finite-dimensional $H$-comodule $V$. 

Let $K\le V\otimes A$ be the kernel of $p$. Like any $H$-equivariant $A$-module, $K$ is a union of objects in $^{fg}\cM_A^H$, i.e. a union of ranges of maps $q_i:W_i\otimes A\to V\otimes A$ for $W_i\in{^f}\cM^H$. These ranges are all in the abelian category $\ol{\cC}$, and they cannot increase indefinitely because every object in $\ol{\cC}$ has finite length by semisimplicity. It follows that $K$ is an object of $\ol{\cC}$, and hence so is $M=(V\otimes A)/K$ (it is simply the kernel of a retract of $V\otimes A$ onto $K$ in the semisimple abelian category $\ol{\cC}$).   

{\bf (b): End of the proof of part (2).} We now have an equivalence 
\begin{equation}\label{eq:free_coco}
 \begin{tikzpicture}[baseline=(current  bounding  box.center),anchor=base,cross line/.style={preaction={draw=white,-,line width=6pt}}]
    \path (0,0) node (1) {${^{fg}}\cM_A^H$} +(3.5,0) node (2) {${^f}\cM^C$}; 
    \draw[->] (1) to node[pos=.5,auto] {$\scriptstyle M\mapsto M/MA^+$} node[pos=.5,auto,swap] {$\simeq$} (2);
  \end{tikzpicture} 
\end{equation}
The whole category $\cM_A^H$ can be recovered from ${^{fg}}\cM_A^H$ as its \define{$\aleph_0$-free cocompletion}. Recall \cite[1.44]{AdaRos94} that this means that $\cM_A^H$ is cocomplete and the inclusion functor ${^{fg}}\cM_A^H\to\cM_A^H$ is $\bC$-linear, right exact, and universal with these properties: Any right exact $\bC$-linear functor from $^{fg}\cM_A^H$ into a cocomplete $\bC$-linear category $\cD$ factors uniquely (up to natural isomorphism) through a cocontinuous functor $\cM_A^H\to \cD$. 

Similarly, $\cM^C$ can be shown to be the $\aleph_0$-free cocompletion of ${^f}\cM^C$. Since $M\mapsto M/MA^+$ is right exact from ${^{fg}}\cM_A^H$ into $\cM^C$, \Cref{eq:free_coco} extends to the desired equivalence $\cM_A^H\simeq \cM^C$.  

{\bf (3): The adjunction \Cref{eq:adj-bis} is an equivalence.} This is immediate from the fact that $C$ is cosemisimple and the last statement in \Cref{th.tak2}: indeed, the cosemisimplicity of $C$ entails the faithful coflatness of the quotient $H\to C$, be it as a left or a right $C$-comodule. 
\end{proof}

\begin{lemma}\label{le.faith}
In the setting of \Cref{th.Tannaka_quot} the functor $\omega:\cC\to \cat{Vec}$ is faithful.  
\end{lemma}
\begin{proof}
Let $V,W$ be finite-dimensional $H$-comodules, so that $V\otimes A$ and $W\otimes A$ are objects of $\cC={^{ff}}\cM_A^H$. We have
\begin{equation}\label{eq:Cmorphisms}
 \hom_{\cC}(V\otimes A,W\otimes A) \cong\hom_{\cM^H}(V,W\otimes A)\cong \hom_{\cM^H}(W^*\otimes V,A). 
\end{equation}
Now define $\cM_H^H$ analogously to $\cM_A^H$; it is the category of right Hopf modules over $H$, and by \cite[Theorem 4.1.1]{Swe69} it is equivalent to $\cat{Vec}$ via $N\mapsto N/NH^+$. Consequently, the functor $\omega:\cC\to\cat{Vec}$ factors as 
\[
 \begin{tikzpicture}[baseline=(current  bounding  box.center),anchor=base,cross line/.style={preaction={draw=white,-,line width=6pt}}]
    \path (0,0) node (1) {$\cC$} +(4,0) node (2) {$\cM_H^H$} +(2,-1.5) node (3) {$\cat{Vec}$}; 
    \draw[->] (1) -- (2) node[pos=.5,auto] {$\scriptstyle M\mapsto M\otimes_AH$};
    \draw[->] (1) to[bend right=10] node[pos=.5,auto,swap] {$\scriptstyle \omega$} (3);
    \draw[->] (2) to[bend left=10] node[pos=.5,auto] {$\scriptstyle N\mapsto N/NH^+$} node[pos=.5,auto,swap] {$\scriptstyle \simeq$}(3);   
  \end{tikzpicture} 
\]
and the horizontal arrow in this diagram acts on hom spaces simply as the inclusion
\[
 \hom_{\cM^H}(W^*\otimes V,A)\le \hom_{\cM^H}(W^*\otimes V,H),
\] 
hence the conclusion. 
\end{proof}

\begin{corollary}\label{cor.Tannaka_quot_expect}
If $\cC={^{ff}}\cM_A^H$ is semisimple the inclusion $A\to H$ splits as a right $A$-module and right $H$-comodule map. 
\end{corollary}
\begin{proof}
Denote by $1\in C$ the image of $1\in H$ through the surjection $H\to C$. It spans a one-dimensional subcomodule $\langle 1\rangle$ of $C$, whose inclusion in $C$ splits because the latter is cosemisimple. The conclusion follows by transporting this situation over to $\cM_A^H$ via the equivalence provided by \Cref{cor.Tannaka_quot}. 
\end{proof}

\begin{remark}\label{re.e}
The retraction $E:H\to A$ in $\cM_A^H$ can be constructed more concretely as follows: Let $h_C:C\to \bC=\langle 1\rangle$ be a right comodule map splitting the inclusion $\langle 1\rangle\le C$, and set
\begin{equation}\label{eq:Tannaka_quot_expect}
  \begin{tikzpicture}[baseline=(current  bounding  box.center),anchor=base,cross line/.style={preaction={draw=white,-,line width=6pt}}]
    \path (0,0) node (1) {$H$} +(2,0) node (2) {$H\otimes H$} +(4.5,0) node (3) {$C\otimes H$} +(6.7,0) node (4) {$H$.}; 
    \draw[->] (1) to node[pos=.5,auto] {$\scriptstyle \Delta$} (2);
    \draw[->] (2) to node[pos=.5,auto] {$\scriptstyle \pi\otimes\id$} (3);
    \draw[->] (3) to node[pos=.5,auto] {$\scriptstyle h_C\otimes\id$} (4);
    \draw[->] (1) to[bend right=20] node[pos=.5,auto,swap] {$\scriptstyle E$} (4);
  \end{tikzpicture}
\end{equation}
It can be shown that the image of $E$ is actually $A$, and $E|_A=\id_A$. 
\end{remark}

\section{A correspondence between coideal $*$-subalgebras and quotient $*$-coalgebras}\label{se.corr}

In this section we apply the results of the previous one to the case when $H$ is a CQG algebra and $\iota:A\to H$ is a right coideal subalgebra closed under $*$. With these assumptions in place throughout the section, the main result reads as follows.

\begin{theorem}\label{th.corr} 
 \begin{enumerate}
\renewcommand{\labelenumi}{(\arabic{enumi})} 
  \item The category ${^{ff}}\cM_A^H$ is semisimple, and hence \Cref{cor.Tannaka_quot} applies. 
  \item $C=H_A$ is a $*$-coalgebra in the sense of \Cref{def.*coalg} such that the quotient map $\pi:H\to C$ intertwines the $*$-structure of $C$ and the map $x\mapsto (Sx)^*$ on $H$.
  \item The coideal subalgebra $A$ can be recovered as ${^C}H$ 
 \end{enumerate}
\end{theorem}
\begin{proof}
{\bf (1)} Let $V\in{^f}\cM^H$. It is enough to show that the finite-dimensional complex algebra $\End(V\otimes A)$ of endomorphisms in ${^{ff}}\cM_A^H$ can be made into a $C^*$-algebra. 

As in \Cref{eq:Cmorphisms}, note that the endomorphism algebra in question is isomorphic to 
\[
 \hom_{\cM^H}(V,V\otimes A)\cong\hom_{\cM^H}(V^*\otimes V,A).
\]
The composition of maps $f,g$ in $\hom_{\cM^H}(V^*\otimes V,H)$ is  
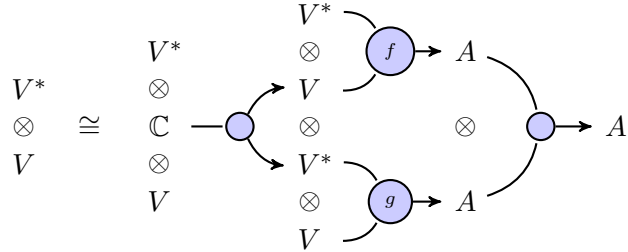
\begin{center}
  \begin{tikzpicture}[>=stealth',shorten >=1pt,auto,node distance=3cm,thick,connect/.style={circle,fill=blue!20,draw,font=\sffamily\small}]  
    \node (11) at (.2,.5) {$V^*$};
    \node (11.12) at (.2,0) {$\otimes^{\ }$};
    \node (12) at (.2,-.5) {$V^{\ }$};

    \node (cong) at (1,0) {$\cong$};

    \node (21) at (2,1) {$V^*$};
    \node (21.22) at (2,.5) {$\otimes^{\ }$};
    \node (22) at (2,0) {$\bC^{\ }$};
    \node (22.23) at (2,-.5) {$\otimes^{\ }$};
    \node (23) at (2,-1) {$V^{\ }$};

    \node[connect] (coev) at (3,0) {};

    \node (31) at (4,1.5) {$V^*$};
    \node (31.32) at (4,1) {$\otimes^{\ }$};
    \node (32) at (4,.5) {$V^{\ }$};
    \node (32.33) at (4,0) {$\otimes^{\ }$};
    \node (33) at (4,-.5) {$V^*$};
    \node (33.34) at (4,-1) {$\otimes^{\ }$};
    \node (34) at (4,-1.5) {$V^{\ }$};

    \node[connect] (f) at (5,1) {$\scriptstyle f$};
    \node[connect] (g) at (5,-1) {$\scriptstyle g$};

    \node (41) at (6,1) {$A$};
    \node (41.42) at (6,0) {$\otimes$};
    \node (42) at (6,-1) {$A$};

    \node[connect] (mult) at (7,0) {};

    \node (51) at (8,0) {$A$};

    \path[]
     (22) edge (coev)
     (coev) edge [->,bend left] (32)
            edge [->,bend right] (33)
     (31) edge [bend left] (f)
     (32) edge [bend right] (f)
     (33) edge [bend left] (g)
     (34) edge [bend right] (g)
     (f) edge [->] (41)
     (g) edge [->] (42)
     (41) edge [bend left] (mult)
     (42) edge [bend right] (mult)
     (mult) edge [->] (51);
 \end{tikzpicture} 
\end{center}
where the tensor products are to be read downwards, morphisms compose rightwards, the leftmost circle is the coevaluation map $\bC\to V\otimes V^*$, and the rightmost one is multiplication. Everything in sight is a morphism in $\cM^H$. 

An inner product on $V$ compatible with the comodule structure in the sense of \Cref{def.compat} can be used to equip the endomorphism algebra $\hom_{\cM^H}(V^*\otimes V,A)$ of $V\otimes A$ in $\cM_A^H$ with a $*$-structure. In order to keep proper track of the construction and not confuse the two instances of $V$, fix another object $W\in{^f}\cM^H$ and endow both $V$ and $W$ endowed with compatible inner products. We will construct a conjugate-linear correspondence
\[
 \hom_{\cM^H}(W^*\otimes V,A)\ni f\leftrightarrow f^*\in \hom_{\cM^H}(V^*\otimes W,A)
\]
as follows. 

Consider an $H$-comodule map $f:W^*\otimes V\to A$. As in \Cref{def.compat}, the compatible inner product on $W$ is just an isomorphism $\ol{W}\cong W^*$ as comodules, where $\ol{W}$ is the complex conjugate comodule defined as in \Cref{eq:conj_comod}. Via this isomorphism, regard $f$ as a map $\ol{W}\otimes V\to A$. 

Now, because the $*$ map on $H$ reverses multiplication, taking conjugate linear comodules reverses the tensor product on the category of comodules. By applying this conjugate linear functor, regard $f$ as a map
\[
 \ol{f}:\ol{V}\otimes W\cong\ol{\left(\ol{W}\otimes V\right)}\to \ol{A}. 
\]   
Finally, $f^*$ is by definition the composition
\begin{equation}\label{eq:f^*}
  \begin{tikzpicture}[baseline=(current  bounding  box.center),anchor=base,cross line/.style={preaction={draw=white,-,line width=6pt}}]
    \path (0,0) node (0) {$V^*\otimes W$} +(2.5,0) node (1) {$\ol{V}\otimes W$} +(4.5,0) node (2) {$\ol{A}$} +(6,0) node (3) {$A$,};
    \draw[->] (0) to node[pos=.5,auto] {$\scriptstyle \cong$} (1);
    \draw[->] (1) to node[pos=.5,auto] {$\scriptstyle \ol{f}$} (2);
    \draw[->] (2) to node[pos=.5,auto] {$\scriptstyle *$} (3);
    \draw[->,bend right] (0) to node[pos=.5,auto,swap] {$\scriptstyle f^*$} (3);
  \end{tikzpicture}
\end{equation}
where the first map is the isomorphism induced by the compatible inner product on $V$, and we have used the fact that $*$ preserves $A$ (and hence constitutes an $H$-comodule map $\ol{A}\to A$).  

Everything above works for $A=H$ too, and gives a conjugate linear correspondence 
\begin{equation}\label{eq:long}
 \hom_{\cat{Vec}}(V,W)\cong \hom_{\cM^H}(W^*\otimes V,H)\ni f\leftrightarrow f^*\in \hom_{\cM^H}(V^*\otimes W,H)\cong \hom_{\cat{Vec}}(W,V),
\end{equation}
where the first and last isomorphisms are applications of Sweedler's theorem \cite[4.1.1]{Swe69} identifying the category of right Hopf modules over $H$ with $\cat{Vec}$ (as in the proof of \Cref{le.faith}). I now claim that this latter correspondence is nothing but the usual adjoint for maps between the Hilbert spaces $V$ and $W$. 

{\bf Sketch of proof for the claim.} This is just a matter of unpacking some definitions, and whatever difficulties there are will be mostly notational. 

Start out with a linear map $t:V\to W$. An examination of the proof of \cite[Theorem 4.1.1]{Swe69} shows that the map $f:W^*\otimes V\to H$ corresponding to $t$ through the first isomorphism in \Cref{eq:long} is
\begin{equation}\label{eq:sketch1}
 w^*\otimes v\mapsto \langle~ (w^*)_0, t(v_0)~\rangle~ (w^*)_1v_1=\langle~ w^*,t(v_0)_0~\rangle~ S(t(v_0)_1)v_1,\ w^*\in W^*,\ v\in V,
\end{equation}
where $\langle -,- \rangle$ denotes the usual pairing between a vector space and its dual and the equality uses the definition of the $H$-comodule structure of $W^*$ in terms of that of $W$.

The identification of $W^*$ with $\ol{W}$ via the compatible inner product $\langle -|-\rangle$ on $W$ turns $w^*$ into some element $w\in W$. Using the compatibility of the inner product of $W$, the right hand side of \Cref{eq:sketch1} becomes 
\begin{equation}\label{eq:sketch2}
 \langle~ w~ |~ t(v_0)_0~\rangle~ S(t(v_0)_1)v_1 = \langle~w_0~|~t(v_0)~\rangle~w_1^*v_1. 
\end{equation}

Starting at the other end of \Cref{eq:long} with the map $t^*:W\to V$ (the Hilbert space adjoint of $t:V\to W$) we get, following the same steps as above, the map 
\begin{equation}\label{eq:sketch3}
 V^*\otimes W\ni v^*\otimes w\mapsto \langle~v_0~|~t^*(w_0)~\rangle~v_1^*w_1, 
\end{equation}
where this time $\langle-|-\rangle$ is the compatible inner product on $V$, and $v\in V=\ol{V}$ is the element corresponding to $v^*\in V^*$ upon identifying the two spaces.  

Since for fixed $v\in V$ and $w\in W$ (with corresponding $v^*\in V^*$ and $w^*\in W^*$ obtained using the compatible inner products) the right hand sides of \Cref{eq:sketch2,eq:sketch3} are obtained from one another by applying the $*$ operation of $H$, it follows that indeed \Cref{eq:sketch1,eq:sketch3} correspond to each other through the construction 
\[
 \hom_{\cM^H}(W^*\otimes V,H)\ni f\leftrightarrow f^*\in \hom_{\cM^H}(V^*\otimes W,H)
\] 
defined as in \Cref{eq:f^*} (with $H$ instead of $A$).

{\bf End of proof of part (1).} Moving back to the original case $W=V$, the claim shows that \Cref{eq:f^*} is in fact a $*$-algebra structure on $\End=\End_{\cM_A^H}(V\otimes A)$, and moreover that the canonical embedding 
\[
 \hom_{\cM^H}(V^*\otimes V,A)\le\hom_{\cM^H}(V^*\otimes V,H)
\]
realizes $\End$ as a $*$-subalgebra of the $C^*$-algebra $\hom_{\cM^H}(V^*\otimes V,H)\cong \End(V)$. Consequently, $\End$ is a finite-dimensional $C^*$-algebra in its own right and hence semisimple.   

{\bf (2)} This is an application of the Tannakian formalism reviewed in \Cref{subse.prel_Tannaka}. 

According to \Cref{th.Tannaka_quot} (and in the notation from its statement), the map $H\to\cat{coend}(\omega)$ resulting from the 2-commutative diagram
\begin{equation}\label{eq:thcorr_(1)}
 \begin{tikzpicture}[baseline=(current  bounding  box.center),anchor=base,cross line/.style={preaction={draw=white,-,line width=6pt}}]
    \path (0,0) node (1) {$ {^f}\cM^H$} +(4,0) node (2) {$ {^{ff}\cM_A^H}$} +(2,-1) node (3) {$ \cat{Vec}$}; 
    \draw[->] (1) -- (2) node[pos=.5,auto]{$\scriptstyle V\mapsto V\otimes A$};
    \draw[->] (1) to[bend right=10] node[pos=.5,auto,swap] {$\scriptstyle \text{forget}$} (3);
    \draw[->] (2) to[bend left=10] node[pos=.5,auto] {$\scriptstyle \omega$} (3);   
  \end{tikzpicture}
\end{equation}
can be identified with $H\to C$. Moreover, the proof of part (1) of the present theorem shows how to promote \Cref{eq:thcorr_(1)} to a 2-commutative diagram 
\begin{equation}\label{eq:thcorr_(1)_bis}
 \begin{tikzpicture}[baseline=(current  bounding  box.center),anchor=base,cross line/.style={preaction={draw=white,-,line width=6pt}}]
    \path (0,0) node (1) {$ {^f}\cM^H$} +(4,0) node (2) {$ {^{ff}\cM_A^H}$} +(2,-1) node (3) {$ {^f}\cat{Hilb}$}; 
    \draw[->] (1) -- (2) node[pos=.5,auto]{$\scriptstyle V\mapsto V\otimes A$};
    \draw[->] (1) to[bend right=10] node[pos=.5,auto,swap] {$\scriptstyle \text{forget}$} (3);
    \draw[->] (2) to[bend left=10] node[pos=.5,auto] {$\scriptstyle \omega$} (3);   
  \end{tikzpicture}
\end{equation}
of $*$-categories and $*$-preserving functors: Simply choose a compatible inner product for each simple $H$-comodule. This will then make both ${^f}\cM^H$ and ${^{ff}}\cM_A^H$ into $*$-categories, the former in the obvious way and the latter via \Cref{eq:f^*}.

Now the discussion preceding \Cref{def.*coalg} says that $*$-functors into ${^f}\cat{Hilb}$ give rise to $*$-coalgebra structures on coendomorphism coalgebras. The $*$-coalgebra structure resulting from the left hand diagonal arrow in \Cref{eq:thcorr_(1)} is just $x\mapsto(Sx)^*$, and the result follows from the fact that the functor ${^f}\cat{Hilb}\to{^f}\cat{Vec}$ is an equivalence (so that the two $\omega$'s in \Cref{eq:thcorr_(1),eq:thcorr_(1)_bis} give rise to canonically isomorphic coendomorphism coalgebras).  

{\bf (3)} We know from \Cref{cor.Tannaka_quot} that \Cref{eq:adj} is an equivalence. The rightward arrow in \Cref{eq:adj} sends $A\in \cM_A^H$ onto the one-dimensional comodule $\langle 1\rangle\in \cM^C$ spanned by the multiplicative unit $1\in H$. This in turn means that $\langle 1\rangle\square_CH=A$, i.e. ${^C}H=A$.
\end{proof}

As a consequence, we have the correspondence promised in the title of this section.

\begin{theorem}\label{th.corr2}
The correspondence \Cref{eq:Galois} restricts to a bijection   
   \begin{equation}\label{eq:corr_Galois}
    \begin{tikzpicture}[auto,baseline=(current  bounding  box.center)]
      \node[text width=2.5cm] (1) at (0,0) {right coideal $*$-subalgebras};
     \node[text width=4cm] (2) at (8,0) {quotient left module $*$-coalgebras.};
     \draw[->] (1) .. controls (2,1) and (6,1) .. node{$A\mapsto H_A$} (2);
     \draw[<-] (1) .. controls (2,-1) and (6,-1) .. node[below] {${^C}H\mapsfrom H$} (2);
    \end{tikzpicture}
   \end{equation}
\end{theorem}
\begin{proof}
First let $\iota:A\to H$ be a right coideal $*$-subalgebra, and $C=H_A$ the corresponding quotient coalgebra. We know from part (3) of \Cref{th.corr} that $A={^C}H$ and hence the left-based loop in \Cref{eq:corr_Galois} is the identity.

Now let $\pi:H\to C$ be an arbitrary left module quotient $*$-coalgebra (in the sense that $\pi$ intertwines $H\ni x\mapsto (Sx)^*$ and the $*$-structure of $C$). 

{\bf (a): $C$ is cosemisimple.} Like any coalgebra over any field, $C$ is the union of its finite-dimensional subcoalgebras. Moreover, we can assume these are preserved by the $*$ operation. So let $C'\le C$ be a finite-dimensional $*$-subcoalgebra. It is a quotient of some finite-dimensional $*$-subcoalgebra $H'\le H$. By the Peter-Weyl theorem for compact quantum groups (\cite[Theorem 3.2.2]{KusTus99}), $H$ is the direct sum of its simple subcoalgebras $C_\alpha$.

The coalgebras $C_\alpha$ are indexed by the simple comodules $V_\alpha$ of $H$, and the dual $C_\alpha^*$ can be identified naturally with the endomorphism algebra $\End(V_\alpha)$. The unique (up to positive scaling) compatible inner product on $V_\alpha$ makes $C_\alpha^*$ into a $*$-algebra and hence $C_\alpha$ into a $*$-coalgebra; the $*$-coalgebra structure $x\mapsto (Sx)^*$ on $H$ is obtained by putting together all of these $*$-coalgebra structures on the individual $C_\alpha$. 

Now, $H'$ is a finite direct sum of $C_\alpha$'s, and hence the dual $(C')^*$ is a $*$-subalgebra of the finite-dimensional $C^*$-algebra $(H')^*$. In particular, $(C')^*$ is semisimple, and hence $C'$ is cosemisimple. Since $C'\le C$ was an arbitrary finite-dimensional subcoalgebra, $C$ itself must be cosemisimple. 

{\bf (b): $A={^C}H$ is a $*$-subalgebra of $H$.} Let $A'$ be the set consisting of all $a^*$ for $a\in A$; it is again a right coideal subalgebra of $H$. Let $C'=H_{A'}$ be the corresponding quotient coalgebra of $H$. \Cref{le.aux} below says that there are mutually inverse conjugate-linear comultiplication-reversing maps $C\leftrightarrow C'$ making 
\begin{equation}\label{eq:desc*coalg}
 \begin{tikzpicture}[baseline=(current  bounding  box.center),anchor=base,cross line/.style={preaction={draw=white,-,line width=6pt}}]
    \path (0,0) node (1) {$ H$} +(4,0) node (2) {$ H$} +(0,-2) node (3) {$ C$} +(4,-2) node (4) {$ C'$} +(2,0) node (5) {$\scriptstyle x\mapsto (Sx)^*$}; 
    \draw[->] (1) to[bend right=20] (2);
    \draw[->] (2) to[bend right=20] (1);
    \draw[->] (1) -- (3);
    \draw[->] (2) -- (4);
    \draw[->] (3) to[bend right=10] (4);
    \draw[->] (4) to[bend right=10] (3);   
  \end{tikzpicture}
\end{equation}
commute. 

Because we already know $x\mapsto (Sx)^*$ descends to a $*$-structure on $C$, we must have $C=C'$. But then $A'$ is contained in ${^C}H=A$ and vice versa. This means $A=A'$, which is exactly what we wanted to show. 

{\bf (c): End of the proof.} Since $C$ is cosemisimple, we have $C=H_A$ by \cite[Theorem 2]{Tak79}. This shows that the loop in \Cref{eq:Galois} based at any quotient $*$-coalgebra is the identity. Since by part (b) above this loop factors through a coideal $*$-subalgebra, the right-based loop in \Cref{eq:corr_Galois} makes sense and is the identity.   
\end{proof}

\begin{lemma}\label{le.aux}
Let $H$ be a CQG algebra, $A\le H$ a right coideal subalgebra, and $A'=\{a^*\ |\ a\in A\}$. Let also $C=H_A$ and $C'=H_{A'}$. Then, $x\mapsto(Sx)^*$ descends to mutually inverse maps $C\leftrightarrow C'$. 
\end{lemma}

We first introduce a piece of notation. In general, for any $\bC$-linear category $\cD$, let $\ol{\cD}$ be the category with the same objects as $\cD$ but conjugate-linear hom spaces. This means that for objects $x,y\in\cD$ $\hom_{\ol{\cD}}(x,y)$ is the same as $\hom_{\cD}(x,y)$ as an underlying abelian group, but the action of $\bC$ on it is twisted by complex conjugation.

\begin{proof of le.aux}
Let $\cC={^{ff}}\cM_A^H$ and $\cC'={^{ff}}\cM_{A'}^H$. Assume we have fixed, once and for all, compatible inner products on all simple $H$-comodules (and hence on all comodules by taking direct sums). 

Imitating the proof of part (1) of \Cref{th.corr}, we define a functor $F:\ol{\cC}^{\mathrm{op}}\to\cC'$ (i.e. a contravariant conjugate-linear functor $\cC\to\cC'$) as follows.
 
For comodules $V,W\in{^f}\cM^H$ identify $\hom_{\cC}(V\otimes A,W\otimes A)$ with $\hom_{\cM^H}(W^*\otimes V,A)$ as before, and similarly with $A'$ and $\cC'$ substituted for $A$ and $\cC$ respectively. The functor $F$ simply sends $V\otimes A$ to $V\otimes A'$. 

Now recall that we have equipped all $H$-comodules with compatible inner products, and are in fact working with unitary comodules. Now let $f:W^*\otimes V\to A$ be an $H$-comodule map (and hence a  morphism in $\cC$). Just as in the proof of \Cref{th.corr}, $f$ can be regarded as a map  
\[
 \ol{f}:\ol{V}\otimes W\cong\ol{\left(\ol{W}\otimes V\right)}\to \ol{A}, 
\]
and $F(f)$ is then defined as 
\begin{equation}\label{eq:F(f)}
  \begin{tikzpicture}[baseline=(current  bounding  box.center),anchor=base,cross line/.style={preaction={draw=white,-,line width=6pt}}]
    \path (0,0) node (0) {$V^*\otimes W$} +(2.5,0) node (1) {$\ol{V}\otimes W$} +(4.5,0) node (2) {$\ol{A}$} +(6,0) node (3) {$A'$.};
    \draw[->] (0) to node[pos=.5,auto] {$\scriptstyle \cong$} (1);
    \draw[->] (1) to node[pos=.5,auto] {$\scriptstyle \ol{f}$} (2);
    \draw[->] (2) to node[pos=.5,auto] {$\scriptstyle *$} (3);
    \draw[->,bend right] (0) to node[pos=.5,auto,swap] {$\scriptstyle F(f)$} (3);
  \end{tikzpicture}
\end{equation}
Now, denoting once more by $\omega:\cC\to{^f}\cat{Hilb}$ the functor from the proof of part (2) of \Cref{th.corr} and by $\omega':\cC'\to{^f}\cat{Hilb}$ its analogue, we have a 2-commutative diagram 
\begin{equation}\label{eq:CC'}
 \begin{tikzpicture}[baseline=(current  bounding  box.center),anchor=base,cross line/.style={preaction={draw=white,-,line width=6pt}}]
   \path (0,0) node (1) {$ \ol{\cC}^{\mathrm{op}}$} +(3,0) node (2) {$ \cC'$} +(.2,-1.3) node (3) {$\scriptstyle
     \ol{{^f}\cat{Hilb}}^{\mathrm{op}}$} +(1.5,-2) node (4) {$\scriptstyle {^f}\cat{Hilb}$}; 
    \draw[->] (1) -- (2) node[pos=.5,auto]{$\scriptstyle F$};
    \draw[->] (1) to[bend right=15] node[pos=.3,auto,swap] {$\scriptstyle \ol{\omega}^{\mathrm{op}}$} (3);
    \draw[->] (3) to[bend right=15] node[pos=.5,auto,swap] {$\scriptstyle *$} (4);
    \draw[->] (2) to[bend left=30] node[pos=.5,auto] {$\scriptstyle \omega'$} (4);
  \end{tikzpicture}
\end{equation}
Here, $\ol{\omega}^{\mathrm{op}}$ is just $\omega$, regarded as a functor between the opposite, conjugate-linear categories and the lower horizontal $*$ is the $*$-structure on $\cat{Hilb}$.

It is not hard to check that the coendomorphism coalgebra of the left hand downward composition $*\circ\ol{\omega}^{\mathrm{op}}$arrow in \Cref{eq:CC'} is $\ol{C}^{\mathrm{cop}}$, i.e. $C$ as an abelian group, with conjugate linear structure as a complex vector space, and with reversed comultiplication as compared to $C$. The diagram \Cref{eq:CC'}, in other words, induces a conjugate-linear, comultiplication-reversing map $C\to C'$. It is compatible with the surjections $H\to C$ and $H\to C'$ because \Cref{eq:CC'} can be embedded in the 2-commutative diagram
\begin{equation}\label{eq:CC'_bis}
 \begin{tikzpicture}[baseline=(current  bounding  box.center),anchor=base,cross line/.style={preaction={draw=white,-,line width=6pt}}]
    \path (0,0) node (1) {$ \ol{\cC}^{\mathrm{op}}$} +(3,0) node (2) {$ \cC'$} +(.2,-1.3) node (3) {$\scriptstyle \ol{{^f}\cat{Hilb}}^{\mathrm{op}}$} +(1.5,-2) node (4) {$\scriptstyle {^f}\cat{Hilb}$} +(-2,.5) node (a) {$ \ol{{_u^f}\cM^H}^{\mathrm{op}}$} +(5,.5) node (b) {$ {_u^f}\cM^H$}; 
    \draw[->] (1) -- (2) node[pos=.5,auto,swap]{$\scriptstyle F$};
    \draw[->] (1) to[bend right=15] node[pos=.3,auto,swap] {$\scriptstyle \ol{\omega}^{\mathrm{op}}$} (3);
    \draw[->] (3) to[bend right=15] node[pos=.5,auto,swap] {$\scriptstyle *$} (4);
    \draw[->] (2) to[bend left=30] node[pos=.5,auto,swap] {$\scriptstyle \omega'$} (4);
    \draw[->] (a) to node[pos=.5,auto] {$\scriptstyle *$} (b);
    \draw[->] (a) to[bend right=20] node[pos=.5,auto,swap] {$\scriptstyle \ol{\text{forget}}^{\mathrm{op}}$} (3);
    \draw[->] (b) to[bend left=20] node[pos=.5,auto] {$\scriptstyle \text{forget}$} (4);
    \draw[->] (a) to (1);
    \draw[->] (b) to (2); 
  \end{tikzpicture}
\end{equation}
Finally, the roles of $C$ and $C'$ can be reversed, and hence $x\mapsto (Sx)^*$ descends to a map $C'\to C$ as well.
\end{proof of le.aux}

To summarize payload of the present section, we have

\begin{corollary}\label{cor.cncl}
  Let $H$ be a CQG algebra and $\iota:A\to H$ a right coideal $*$-subalgebra. Then:
  \begin{enumerate}
    \renewcommand{\labelenumi}{(\arabic{enumi})}
  \item $\iota$ splits as a right $A$-module, right $H$-comodule map.
  \item $C=H_A$ is a cosemisimple left module quotient $*$-coalgebra.
  \item $A={}^CH$.
  \item \Cref{eq:adj,eq:adj-bis} are equivalences.
  \item $H$ is faithfully flat as a right $A$-module.     
  \end{enumerate}
\end{corollary}
\begin{proof}
  Part (1) of \Cref{th.corr} says that under the present hypotheses \Cref{cor.Tannaka_quot} applies, and the latter implies claims (2) and (4). Claim (3) is part of \Cref{th.corr}, and (1) follows from \Cref{cor.Tannaka_quot,cor.Tannaka_quot_expect}.

  Finally, claim (5) follows from the fact that \Cref{eq:adj-bis} is an equivalence together with \Cref{th.tak2}.
\end{proof}

Note that part (5) of \Cref{cor.cncl} is, in principle, chiral (i.e. left-right asymmetric): $A$ is a {\it right} coideal subalgebra, and we are asserting {\it right} faithful flatness. Nevertheless, we can symmetrize the flatness result.

\begin{corollary}\label{cor.left-fl}
  Under the hypotheses of \Cref{cor.cncl} $H$ is also left faithfully flat over $A$
\end{corollary}
\begin{proof}
  Given the right faithful flatness in part (5) of \Cref{cor.cncl}, the conclusion follows from the simple general observation in \Cref{le.lr-flat} below.
\end{proof}

\begin{lemma}\label{le.lr-flat}
  For a morphism $\iota:A\to H$ of $*$-algebras left and right flatness are equivalent; the same goes for faithful flatness. 
\end{lemma}
\begin{proof}
  Consider the autoequivalence of $\cat{Vec}$ sending a vector space $V$ to its complex conjugate $\overline{V}$. It extends to an equivalence
  \begin{equation*}
    F:{}_A\cM\to \cM_A
  \end{equation*}
  between the categories of left and right $A$-modules: given a left $A$-module $M$, its complex conjugate $\overline{M}$ can be equipped with the right $A$-module structure $\triangleleft$ defined by
  \begin{equation*}
    m\triangleleft a:= a^*m,\ \forall m\in M,\ \forall a\in A. 
  \end{equation*}
  $F$ intertwines the endofunctors $H\otimes_A-$ and $-\otimes_AF(H)$, so that one of these functors is (faithfully) exact if and only if the other is. Finally, it remains to observe that the $*$-structure $x\mapsto x^*$ on $H$ implements an isomorphism between $H$ viewed as a left $A$-module and $F(H)$ (the latter with $H$ regarded as a right $A$-module).
\end{proof}

\section{Expectations and Fourier transforms}\label{se.exp}

According to part (1) of \Cref{cor.cncl} we have an ``expectation'' $E:H\to A$ splitting the inclusion $\iota:A\to H$ both as morphism in $\cM_A^H$ (i.e. as a map of right $A$-modules and right $H$-comodules). As discussed in the introduction, we are interested in characterizing those $\iota:A\to H$ for which $E$ is positive as described above. Recall our use of the term:

\begin{definition}\label{def.pos}
Let $H$ be a $*$-algebra admitting a $C^*$ completion, $\iota:A\to H$ a $*$-subalgebra, and $E:H\to A$ a linear map. We say that $E$ is {\it completely positive} (or occasionally just {\it positive}) if \Cref{eq:pos} holds for all $x_i\in H$ and $a_i\in A$ in the completion $A_u$ of $A$ with respect to the universal $C^*$ norm on $H$.  
\end{definition}

The main result of this section is

\begin{theorem}\label{th.s2}
  Let $\iota:A\to H$ be a right coideal $*$-subalgebra of a CQG algebra, and $E:H\to A$ the right $A$-module, right $H$-comodule splitting of $\iota$ from \Cref{cor.cncl}.

  $E$ is positive in the sense of \Cref{def.pos} if and only if $A$ is invariant under the square of the antipode $S=S_H$. 
\end{theorem}

Before embarking on the proof we need some preparation. For at least part of the argument we will follow the strategy employed in \cite[$\S$3]{ff}, and hence will refer to that earlier material quite extensively. 

Recall from the introduction that when, occasionally, we want to complete $H$ to a $C^*$-algebra we do so with respect to the maximal or universal norm, and
denote the completion by $H_u$.

As a first step towards proving \Cref{th.s2}, we reduce the positivity of $E$ to that of a functional $\varphi:H\to \bC$ (recall that such a functional is {\it positive} if $\varphi(x^*x)\ge 0$ for all $x\in H$). Denote
\begin{equation*}
 \begin{tikzpicture}[baseline=(current  bounding  box.center),anchor=base,cross line/.style={preaction={draw=white,-,line width=6pt}}]
    \path (0,0) node (1) {$H$} +(2,.5) node (2) {$C$} +(4,0) node (3) {$\bC,$}; 
    \draw[->] (1) to[bend left=6] node[pos=.5,auto] {$\scriptstyle \pi$} (2);
    \draw[->] (2) to[bend left=6] node[pos=.5,auto] {$\scriptstyle h_C$} (3);
    \draw[->] (1) to[bend right=10] node[pos=.5,auto,swap] {$\scriptstyle \varphi$} (3);
  \end{tikzpicture}
\end{equation*}
where $h_C$ is the unique unital integral on $C$, in the sense that $h_C(\pi(1))=1$; the existence of $h_C$ follows from the cosemisimplicity of $C$ (\Cref{cor.cncl} (2)).

The following observation simply recasts \cite[Lemma 3.4]{ff}, which only differs in that $\iota:A\to H$ there was assumed to be a Hopf subalgebra. The proof applies verbatim to the somewhat more general setting here.

\begin{lemma}\label{le.phi}
  Under the hypotheses of \Cref{th.s2}, $E:H\to A$ is positive in the sense of \Cref{def.pos} if and only if $\varphi:H\to \bC$ defined above is positive.
  \qedhere
\end{lemma}

As in \cite[$\S$3]{ff}, we write
\begin{equation*}
  C=\bigoplus_\alpha C_\alpha
\end{equation*}
for the decomposition of $C$ as a direct sum of matrix coalgebras $C_\alpha\cong M_{n_\alpha}^*$ (dual of the $n_\alpha\times n_\alpha$ matrix algebra $M_{n_\alpha}$). We also denote by $C^\bullet$ the (non-unital) $*$-algebra
\begin{equation*}
  C^\bullet=\bigoplus_\alpha C_\alpha^*. 
\end{equation*}
Recall \cite[Definition 3.10]{ff}

\begin{definition}\label{def.four}
  The {\it $\varphi$-relative Fourier transform} $\cF:H\to C^{\bullet}$ is the map defined by
  \begin{equation*}
    H\ni x\mapsto \varphi(Sx\cdot). 
  \end{equation*}
\end{definition}

The name is meant to recall the analogy to Fourier transforms on locally compact abelian groups (e.g. \cite[Chapter 1]{rud}), turning functions on $G$ into functions on its Pontryagin dual $\widehat{G}$. As mentioned in the discussion following \cite[Definition 3.10]{ff}, the notion is inspired by the non-commutative Fourier transform used extensively in \cite{pw}.

We prove \Cref{th.s2} in stages. Consider first one of the implications. 

\begin{proposition}\label{pr.s2>pos}
Under the hypotheses of \Cref{th.s2}, suppose $S^2(A)=A$. Then, $E$ is positive. 
\end{proposition}
\begin{proof}
  Under the assumption that $S^2(A)=A$ the squared antipode descends to an automorphism of $C=H/HA^+$. We can now repeat the proof of \cite[Proposition 3.8]{ff} to conclude that if $e_i$ and $e^i$, $i\in I$ are dual bases of $C$ and $C^\bullet$ respectively, then
  \begin{equation*}
    \theta=\sum_i \langle e^i,S^2e_{i(1)}\rangle e_{i(2)}\in \overline{C}:=\prod_\alpha C_\alpha
  \end{equation*}
  (whose definition makes sense because $S^2$ acts on $C$) is a positive functional on $C^\bullet$.

  Moreover, the proof of \cite[Theorem 3.1]{ff} now goes through to show that for any choice of $x,y\in H$ we have
  \begin{equation*}
    \varphi\left(\left(S^2y\right)^* S^2x\right) = \theta\left(\left(\cF y\right)^*\cF x\right). 
  \end{equation*}
  The positivity of $\theta$ now implies that of $\varphi$, given that $S^2:C\to C$ is onto.
\end{proof}

\Cref{pr.s2>pos} takes care of half of \Cref{th.s2}. before proving the converse, we need to recall some background on von Neumann algebra modular theory, as covered in \cite{tak2}. Consider the von Neumann algebra $M$ obtained as the $W^*$ closure of the GNS representation of $H$ on the Hilbert space $L^2(H,h)$ attached to the Haar state $h$. Denote also by $N$ von Neumann closure of $A$ in $M$.

Under the assumption that $E:H\to A$ is positive, it extends to an expectation $E:M\to N$ with respect to the faithful state $h$ in the sense of \cite[Definition IX.4.1]{tak2}, as recounted in the introduction. According to (one implication of) \cite[Theorem IX.4.2]{tak2}, this means that $N$ is invariant under the {\it modular automorphism group} $\sigma_t$, $t\in \bR$ of $M$ attached to the weight $h$. Recall what this means \cite[$\S$VIII.1]{tak2}:  

There is a unique one-parameter group $\sigma_t$ of von Neumann algebra automorphisms of $M$ such that
\begin{itemize}
\item $h\circ \sigma_t=h$ for all $t\in \bR$; 
\item For every pair of elements $x,y\in M$ there is a bounded continuous function $F_{x,y}$ on the strip $\{\Im(z)\in [0,1]\}$, holomorphic in the interior of the strip, such that
  \begin{equation*}
    F_{x,y}(t)=h(\sigma_t(x)y),\quad F_{x,y}(t+i)=h(y\sigma_t(x)),\ \forall t\in \bR. 
  \end{equation*}
\end{itemize}

Now, according to \cite[Theorem 5.6]{Wor87}, the modular automorphism group of $h$, when restricted to $H$, is of the form
\begin{equation*}
  H\ni x\mapsto \sigma_t(x)=f_{ti}*x*f_{ti},
\end{equation*}
where
\begin{itemize}
\item $\{f_z\}_{z\in \bC}$ is a family of functionals on $H$ such that $z\mapsto f_z(x)$ is entire and of exponential growth on the right hand half plane for every $x\in H$;   
\item `$*$' signifies convolution of functionals on and elements of $H$, as in
  \begin{equation*}
    f*x = f(x_{(2)})x_{(1)},\quad x*f = f(x_{(1)})x_{(2)};
  \end{equation*}
\item $S^2(x)=f_{-1}*x*f_1$ for all $x\in H$. 
\end{itemize}

With this bit of background in hand, we can proceed.

\begin{proof of s2}
As noted, we already have one implication by \Cref{pr.s2>pos}, so it suffices to prove its converse. To this end, suppose $E:H\to A$ is positive as per \Cref{def.pos}.

As per the above discussion, we have
\begin{equation*}
  f_{ti}*x*f_{ti}\in A,\ \forall x\in A,\ \forall t\in \bR
\end{equation*}
for the family of functionals $\{f_z\}_{z\in \bC}$ constructed in \cite[Theorem 5.6]{Wor87}.

Fix an arbitrary $x\in A$, and let $C\subset H$ be a finite-dimensional subcoalgebra containing $x$. Consider the map
\begin{equation*}
  \psi:\bC\to C/C\cap A
\end{equation*}
sending $z$ to the image of $f_z*x*f_z$ (note that left and right convolutions by functionals preserve $C$, as the latter is a subcoalgebra). The map is holomorphic if we give the complex finite-dimensional vector space $C/C\cap A$ its standard complex manifold structure, and we know that $\psi$ vanishes on the imaginary axis. It follows that it vanishes everywhere, i.e.
\begin{equation}
  \label{eq:zz}
  f_z*x*f_z\in A,\ \forall z\in \bC. 
\end{equation}
On the other hand, because $A$ is a right $H$-coideal, we have
\begin{equation}
  \label{eq:zl}
  f_w*x\in A,\ \forall w\in \bC. 
\end{equation}
Since \Cref{eq:zz,eq:zl} apply to all $x\in A$ and by \cite[Theorem 5.6]{Wor87} $f_{w}*f_z=f_{w+z}$, we have
\begin{equation*}
  f_u * x * f_v\in A,\ \forall x\in A,\ \forall u,v\in \bC. 
\end{equation*}
In particular
\begin{equation*}
  S^2(x)=f_{-1}*x*f_1\in A,\ \forall x\in A
\end{equation*}
and similarly for $S^{-2}(x)=f_1*x*f_{-1}$, finishing the proof. 
\end{proof of s2}


\bibliographystyle{plain}
\addcontentsline{toc}{section}{References}

\def\polhk#1{\setbox0=\hbox{#1}{\ooalign{\hidewidth
  \lower1.5ex\hbox{`}\hidewidth\crcr\unhbox0}}}

\end{document}